\pdfoutput=1

    \documentclass[11pt,twoside]{article}

    \usepackage[kerning,spacing]{microtype}
    \microtypecontext{spacing=nonfrench}

    \usepackage{graphicx}
    \usepackage{styleset}
    \usepackage{macros}
    \let\subsubsection\subparagraph

    \title  {Gauge theory and Rasmussen's invariant}

    \author {P. B. Kronheimer and T. S. Mrowka%
            \thanks{%
            The work of the first author
            was supported by the National Science Foundation through
            NSF grant DMS-0904589. The work of
            the second author was supported by NSF grant DMS-0805841.}} 

    \address {Harvard University, Cambridge MA 02138 \\
              Massachusetts Institute of Technology, Cambridge MA 02139}       

\begin{document}

\maketitle

\section{Introduction}

For a knot $K \subset S^{3}$, the (smooth) \emph{slice-genus}
$g_{*}(K)$ is the smallest genus of any properly embedded,
smooth, oriented surface $\Sigma\subset B^{4}$ with boundary $K$.  In
\cite{Rasmussen}, Rasmussen used a construction based on Khovanov
homology to define a knot-invariant $s(K)\in 2\Z$ with the following
properties:
\begin{enumerate}
\item $s$ defines a homomorphism from the knot concordance group to $2\Z$;
\item $s(K)$ provides a lower bound for the slice-genus a knot $K$, in that
        \[
                  2 g_{*}(K) \ge s(K);
         \]
\item the above inequality is sharp for positive knots (i.e. knots
    admitting a projection with only positive crossings).
\end{enumerate}
These results from \cite{Rasmussen} were notable because they
provided, for the first time, a purely combinatorial proof of an
important fact of 4-dimensional topology, namely that the slice-genus
of a knot arising as the link of a singularity in a complex plane
curve is equal to the genus of its Milnor fiber.  This result, the
local Thom conjecture, had been first proved using gauge theory, as an
extension of Donaldson's techniques adapted specifically to the
embedded surface problem \cite{GTES-I, GTES-II}.

The purpose of this paper is to exhibit a very close connection
between the invariant $s$ and the constructions from gauge theory that
had been used earlier. Specifically, we shall use gauge theory to
construct an invariant $\ssharp(K)$ for knots $K$ in $S^{3}$ 
in a manner that can be made to look structurally very similar to a
definition of Rasmussen's $s(K)$.  We shall then prove that these invariants are
in fact equal:
\begin{equation}
    \label{eq:1}
    \ssharp(K) = s(K),
\end{equation}
for all knots $K$. The construction of $\ssharp$ is based on the
instanton homology group $\Isharp(K)$ as defined in
\cite{KM-unknot}. The equality of $\ssharp$ and $s$ is derived from
the relationship between instanton homology and Khovanov homology
established there. A knot invariant very similar to $\ssharp$ appeared
slightly earlier in \cite{KM-singular}, and its construction derives
from an approach to the minimal-genus problem developed in
\cite{Obstruction}.

Because it is defined using gauge theory, the invariant $\ssharp$ has
some properties that are not otherwise manifest for $s$. We shall
deduce:

\begin{corollary}\label{cor:neg-def}
    Let $K$ be a knot in $S^{3}$, let $X^{4}$ be an oriented
    compact 4-manifold with boundary $S^{3}$, and let $\Sigma \subset
    X^{4}$ be a properly embedded, connected oriented surface with
    boundary $K$. Suppose that $X^{4}$ is negative-definite and has
    $b_{1}(X^{4})=0$. Suppose also that
    the inclusion of $\Sigma$ is homotopic relative to $K$ to a map to the
    boundary $S^{3}$. Then
      \[
                 2 g(\Sigma) \ge s(K).
       \]
     In particular, this inequality holds for any such surface $\Sigma$
     in a homotopy 4-ball.
\end{corollary}

The smooth 4-dimensional Poincar\'e conjecture is equivalent to the
statement that a smooth homotopy 4-ball with boundary $S^{3}$ must be a
standard ball. A possible approach to showing that this conjecture is
false is to seek a knot $K$ in $S^{3}$ that bounds a smooth disk in
some homotopy ball but not in the standard ball. To carry this
through, one needs some way of showing that a suitably-chosen 
$K$ does not bound a disk
in the standard ball, by an argument that does not apply equally to
homotopy balls. This approach was advanced in \cite{FGMW}, with the
idea of using the non-vanishing of $s(K)$ for this key step. The
corollary above shows that Rasmussen's invariant $s$ cannot be used for this
purpose: a knot $K$  cannot bound a smooth disk
in any homotopy-ball if $s(K)$ is non-zero.

\subparagraph{Acknowledgements.}  The authors would like to thank the
Simons Center for Geometry and Physics and the organizers of the
workshop \emph{Homological Invariants in Low-Dimensional Topology} for
their hospitality during the preparation of these results.

\section{The Rasmussen invariant}

\subsection{Conventions for Khovanov homology}
\label{subsec:Khovanov}

Let $\lambda$ be a formal variable, and let $\Ring$ denote the ring
$\Q[[\lambda]]$ of formal power series. Let $V$ be the free rank-2
module over $\Ring$
with a basis $\vp$ and $\vm$, and
let $\mprod$ and $\Delta$ be defined by,
\[
\begin{aligned}
    \mprod : \vp \otimes \vp &\mapsto \vp \\
    \mprod :  \vp \otimes \vm &\mapsto \vm \\
    \mprod:  \vm \otimes \vp &\mapsto \vm \\
    \mprod:  \vm \otimes \vm &\mapsto \lambda^{2} \vp ,
\end{aligned}
\]
\[
\begin{aligned}
    \Delta : \vp  &\mapsto \vp\otimes \vm + \vm\otimes \vp \\
   \Delta : \vm  &\mapsto \vm\otimes \vm + \lambda^{2} \vp\otimes \vp .
\end{aligned}
\]
These definitions coincide with those of $\mathcal{F}_{3}$ from
\cite{Khovanov-2}, except that the variable $t$ from that paper is now
$\lambda^{2}$ and we have passed to the completion by taking formal
power series.

In order to match the conventions which arise naturally from gauge
theory, we will adopt for the duration of this paper a non-standard
orientation convention for Khovanov homology, so that what we call
$\kh(K)$ here will coincide with the Khovanov homology of the mirror
image of $K$ as usually defined. To be specific, let $K$ be an
oriented link with a plane diagram $D$, having $n_{+}$ positive
crossings and $n_{-}$ negative crossings. Write $N=n_{+} + n_{-}$. For
each vertex $v$ of the cube $\{0,1\}^{N}$, let $K_{v}$ be the
corresponding smoothing, with the standard convention about the
``$0$'' and ``$1$'' smoothings, as illustrated in
Figure~\ref{fig:smoothings}. 
\begin{figure}
    \centering
    \includegraphics[height=1.4 in]{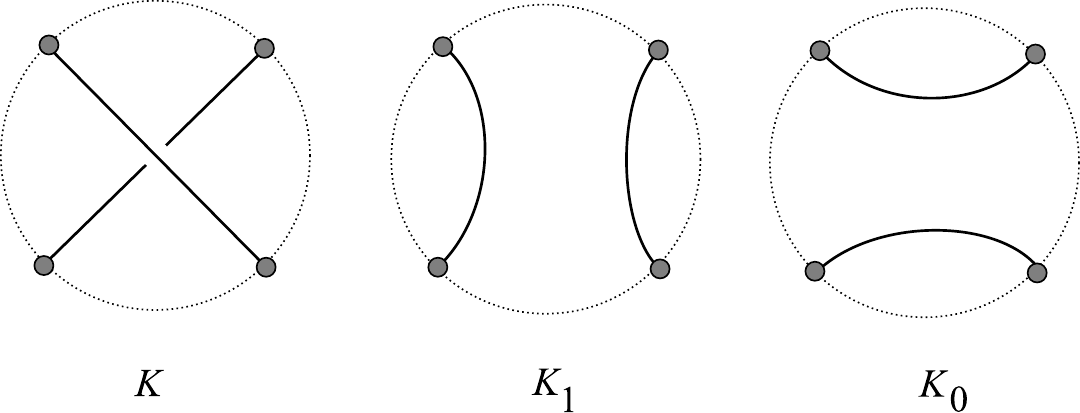}
    \caption{The $0$- and $1$- smoothings at a crossing of $K$.}
    \label{fig:smoothings}
\end{figure}
 For each vertex $v$ of the cube let
$C_{v} = C(K_{v})$ be the tensor product $V^{\otimes p(v)}$, where the
factors are indexed by the components of the unlink $K_{v}$ and the
tensor product is over $\Ring$.  Suppose that $v > u$ and that $(v,u)$
is an edge of the cube: so the coordinates of $v$ and $u$ differ in
exactly one spot at which $v$ is $1$ and $u$ is $0$. Define $d_{vu} :
C_{v} \to C_{u}$ to be the map
\[
        d_{vu} =
        \begin{cases}
            \epsilon_{vu} \mprod, & p(v) = p(u) + 1 ,\\
            \epsilon_{vu} \Delta, & p(v) = p(u) - 1 ,
        \end{cases}
\]
where $\epsilon_{vu}$ is $\pm 1$ and is defined by the conventions of
\cite{Khovanov-1}. Set
\[
\begin{aligned}
    \bC(D) &= \bigoplus_{v} C_{v} \\
    d &= \bigoplus_{v > u} d_{vu}
\end{aligned}
\]
where the second sum is over the edges of the cube. The essential
difference here from the standard conventions is that our differential
$d$ runs from ``$1$'' to ``$0$'' along each edge of the cube, rather
than the other way.

We define a quantum grading and a cohomological grading on $\bC(D)$ as
follows. We assign $\vp$ and $\vm$ quantum gradings $+1$ and $-1$
respectively, we assign $\lambda$ a quantum grading $-2$, 
and for each vertex $v$ we use these to define a grading $Q$ on the tensor product
$C_{v} = V^{\otimes p(v)}$ as usual. We then define the quantum
grading $q$ on $C_{v}$ by
\[
              q = Q - |v| - n_{+} + 2n_{-}
\]
where $|v|$ denotes the sum of the coordinates. We define the
cohomological grading on $C_{v}$ by
\[
              h = -|v| + n_{-}.
\]
With these conventions, $d:\bC(D) \to \bC(D)$ has bidegree $(1,0)$
with respect to $(h,q)$. We define
\[
                   \kh(K) = H_{*}(\bC(D) , d)
\]
with the resulting bigrading. If one were to replace $\lambda^{2}$ by
$1$ in the above definitions of $m$ and $\Delta$, one would recover
the Lee theory \cite{Lee}.

We will also make use of a mod $4$ grading on the complex
$\bC(D)$. Define $q_{0}$ in the same manner as $q$ above, but
assigning grading $0$ to the formal variable $\lambda$. Then define
\[
            \gr = q_{0} - h.
\]
With respect to this grading, $d$ has degree $-1$ and the homology of
the unknot is a free module of rank $2$ of the form $\Ring \up \oplus
\Ring \um$, with the two summands appearing in gradings $1$ and $-1$ respectively.

\subsection{Equivariant Khovanov homology and the $s$-invariant}

The $s$-invariant of a knot $K$ is defined in \cite{Rasmussen} using
the Lee variant of Khovanov homology. An alternative viewpoint on the
same construction is given in \cite{Khovanov-2}. We give here a
definition of $s(K)$ modelled on that second approach.

Let $K$ be a knot. As a finitely generated module over the ring $\Ring$, the homology
$\kh(K)$ is the direct sum of a free module and a torsion
module. Write $\kh'(K)$ for the quotient
        \[\kh'(K) = \kh(K)/\mathrm{torsion}.\]
The first key fact, due essentially to Lee \cite{Lee}, is that
$\kh'(K)$ has rank $2$. More specifically, it is supported in
$h$-grading $0$ and has a single summand in the mod-4 gradings $\gr=1$
and $\gr=-1$. We write $\zp$ and $\zm$ for corresponding generators in
these mod-4 gradings. 

Choose now
an oriented cobordism $\Sigma$ from the unknot $U$ to
$K$. After breaking up $\Sigma$ into a sequence of elementary
cobordisms, one arrives at a map of $\Ring$-modules
\[
        \psi(\Sigma) : \kh(U) \to \kh(K)
\]
and hence also a map of torsion-free $\Ring$-modules
\[
        \psi'(\Sigma) : \kh'(U) \to \kh'(K).
\]
This map has degree $0$ or $2$ with respect to the mod $4$ grading,
according as the genus of $\Sigma$ is even or odd respectively.
There is therefore a unique pair of non-negative integers $m_{+}$, $m_{-}$ such
that either
\[
\begin{aligned}
    \psi'(\Sigma)(\up) &\sim \lambda^{m_{+}} \zp \\
    \psi'(\Sigma)(\um) &\sim \lambda^{m_{-}} \zm \\
\end{aligned}
\]
or
\[
\begin{aligned}
    \psi'(\Sigma)(\up) &\sim \lambda^{m_{+}} \zm \\
    \psi'(\Sigma)(\um) &\sim \lambda^{m_{-}} \zp ,
\end{aligned}
\]
according to the parity of the genus $g(\Sigma)$. Here $\sim$ means
that these elements are associates in the ring. Set
\[
            m(\Sigma) = \frac{1}{2} ( m_{+} + m_{-}).
\]
Rasmussen's $s$-invariant can then be defined by
\[
        s(K) = 2 ( g(\Sigma) - m(\Sigma)).
\]
Although our orientation convention for $\kh(K)$ differs from the
usual one, the above definition of $s$ brings us back into line
with the standard signs, so that $s(K) \ge 0$ for positive knots (see
\cite{Rasmussen}). 

We have phrased the definition of $s(K)$ so that there is no up-front
reference to the $q$-grading. If one grants that $s(K)$ is
well-defined (i.e.~that it is independent of the choice of
$\Sigma$), then it is clear from the definition that $s(K)$ is a lower
bound for $2g_{*}(K)$. 

\section{The gauge theory version}

\subsection{Local systems and the definition of $\ssharp$}

Given an oriented link $K$ in $\R^{3}$, let us form the link
$K^{\sharp}$ in $S^{3}$ by adding an extra Hopf link near the point at
infinity. Let $w$ be an arc joining the two components of the Hopf
link. In \cite{KM-unknot}, we defined
\[
     \Isharp(K) = I^{w}(S^{3}, K^{\sharp}).
\]
This is a Floer homology group defined using certain orbifold $\SO(3)$ connections on
$S^{3}$ (regarded as an orbifold with cone-angle $\pi$ along $K^{\sharp}$). The arc $w$ is a
representative for $w_{2}$ of these connections. As defined in
\cite{KM-unknot} the ring of coefficients is $\Z$, but we now wish to
define the group using a non-trivial local system on the space
$\bonf$ of orbifold connections on $S^{3}$, as in \cite{KM-singular}.

We review the construction of the local system. 
Let $\Sring$ denote the ring 
\[
\begin{aligned}
    \Sring&=\Q[u^{-1}, u] \\ &= \Q[\Z].
\end{aligned}
\]
We can regard $\Sring$ as contained in the ring $\Q[\R]$ and use the
notation $u^{x}$ ($x\in\R$) for generators of the larger ring. Given a
continuous function
\[
           \mu : \bonf \to \R/\Z,
\]
we define a local system of free rank-1 $\Sring$-modules by defining
\[\Gamma_{a} = u^{\mu(a)} \Sring\] for $a\in\bonf$. The particular
$\mu$ we wish to use follows the scheme of equation (89) of
\cite{KM-singular}. That is, we frame the link $K$ with the Seifert
framing and for each orbifold connection $A$ we define
\[
        \mu([A]) \in U(1) = \R/\Z
\]
by taking the product of the holonomies of $[A]$ along the components
of $K$. The framing is used to resolve the ambiguity in this
definition in the orbifold context: see \cite{KM-singular}. We write
\[
      \Isharp(K; \Gamma)
\]
for the corresponding Floer homology group.
 Note that
we do not make use of the extra two components belonging the Hopf link in
$K^{\sharp}$ in the definition of $\Gamma$. 
(The definition in this paper is essentially
identical to that in \cite{KM-singular}, except that we were
previously not adding the Hopf link to $K$: instead, we considered
$I^{w}(T^{3} \# (S^{3}, K);\Gamma)$, where $w$ was a standard circle
in the $T^{3}$. The variable $u$ was previously called $t$.)

The homology group $\Isharp(K;\Gamma)$ is naturally $\Z/4$ graded, and
with one choice of standard conventions the homology of the unknot $U$
is given by
\[
          \Isharp(U;\Gamma) = \Sring \up \oplus \Sring \um
\]
where $\up$ and $\um$ have mod-4 gradings $1$ and $-1$ respectively.

The following points are proved in \cite{KM-singular}. First, 
an oriented cobordism $\Sigma \subset [0,1]\times S^{3}$ between links induces a map
$\psi^{\sharp}(\Sigma)$ of their instanton homology groups. In the case of a
cobordism between knots, the map $\psi^{\sharp}(\Sigma)$ as degree $0$ or $2$
mod $4$, according as the genus of $\Sigma$ is even or odd. 
Using a blow-up
construction, this definition can be extended to ``immersed
cobordisms'' $\Sigma \subset [0,1] \times S^{3}$ with normal
crossings. We can then consider how the map $\psi^{\sharp}(\Sigma)$
corresponding to an immersed cobordism $\Sigma$ changes
when we change $\Sigma$ by one of the local moves indicated
schematically in Figure~\ref{fig:moves}. (The figure represents
analogous moves for an immersed arc in $\R^{2}$.)
\begin{figure}
    \centering
    \includegraphics[height=2.0 in]{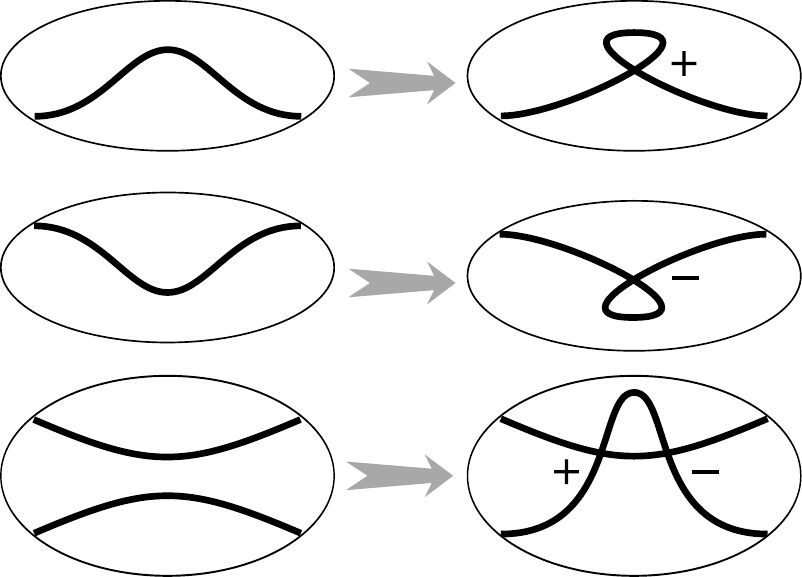}
    \caption{Schematic pictures in lower dimension, representing 
    the positive and negative twist moves and the finger move.}
    \label{fig:moves}
\end{figure}

\begin{proposition}[\protect{\cite[Proposition 5.2]{KM-singular}}]
    Let $\Sigma^{*}$ be obtained from $\Sigma$ by either a positive
    twist move (introducing a positive double-point), a negative twist
    move (introducing a negative double-point), or a finger move
    (introducing a cancelling pair of double-points). Then we have,
    respectively:
    \begin{enumerate}
    \item $\psi^{\sharp}(\Sigma^*) = (1-u^{2})\psi^{\sharp}(\Sigma)$ for the positive twist move;
   \item $\psi^{\sharp}(\Sigma^*) = \psi^{\sharp}(\Sigma)$ for the negative twist move;
   \item $\psi^{\sharp}(\Sigma^*) = (1-u^{2})\psi^{\sharp}(\Sigma)$ for the finger move.
    \end{enumerate} \qed
\end{proposition}

From these properties one can deduce the following. 
For any
knot $K$, the
torsion submodule of the $\Sring$-module $\Isharp(K;\Gamma)$ is
annihilated by some power of $(1-u^{2})$.  Second, just as for the
unknot, the quotient
$
         \Isharp(K;\Gamma)/\mathrm{torsion}
$
is a free module of rank $2$, with generators in mod-4 gradings $1$
and $-1$
\[
         \Isharp(K;\Gamma)/\mathrm{torsion} = \Sring \zp \oplus \Sring \zm.
\]
So if $\Sigma$ is
a cobordism from the unknot $U$ to $K$, then the induced map on
$\Isharp/\mathrm{torsion}$ is determined by a pair of elements
$\sigma_{+}(\Sigma)$, $\sigma_{-}(\Sigma)$ in $\Sring$, in that we have either
\begin{subequations}\label{eq:sigma-pm}
\begin{equation}\label{eq:sigma-pm-1}
\begin{aligned}
    \psi^{\sharp}(\Sigma)(\up) &= \sigma_{+}(\Sigma) \zp \\
    \psi^{\sharp}(\Sigma)(\um) &= \sigma_{-}(\Sigma) \zm 
\end{aligned}
\end{equation}
or
\begin{equation}\label{eq:sigma-pm-2}
\begin{aligned}
    \psi^{\sharp}(\Sigma)(\up) &= \sigma_{+}(\Sigma) \zm \\
    \psi^{\sharp}(\Sigma)(\um) &= \sigma_{-}(\Sigma) \zp 
\end{aligned}
\end{equation}
\end{subequations}
modulo torsion elements, according as the genus of $\Sigma$ is even or
odd. Furthermore, these elements depend only on the genus of
$\Sigma$. The dependence is as follows. If $\Sigma_{1}$ denotes a
surface of genus $1$ larger, then
\begin{equation}\label{eq:p-and-q}
\begin{aligned}
    \sigma_{+}(\Sigma_{1})& = q(u) \sigma_{-}(\Sigma) \\
  \sigma_{-}(\Sigma_{1})& = p(u) \sigma_{+}(\Sigma) 
\end{aligned}
\end{equation}
where $p(u)$ and $q(u)$ are two elements of $\Sring$, independent of
$K$ and $\Sigma$. They are determined by the map $\psi^{\sharp}$ arising from
a genus-1 cobordism from $U$ to $U$. See equation (105) in
\cite{KM-singular}. From that paper we record the fact that
\begin{equation}\label{eq:pq}
            p(u) q(u) = 4 u^{-1} (u-1)^{2}.
\end{equation}
Later in this paper we will compute $p$ and $q$, but for now the above
relation is all that we require.

Let us now pass to the completion of the local ring of $\Sring$ at
$u=1$: the ring of formal power series in the variable $u-1$. In this
ring of formal power series, let $\lambda$ be the solution to the
equation
\[
       \lambda^{2} = u^{-1}(u-1)^{2}
\]
given by
\[
         \lambda = (u-1) - \frac{1}{2}(u-1)^{2} + \cdots,
\]
so that the completed local ring can be identified with the ring
$\Ring = \Q[[\lambda]]$. We can then write
\begin{equation}\label{eq:msharp-def}
\begin{aligned}
    \sigma_{+}(\Sigma) & \sim \lambda^{\msharp_{+}} \\
    \sigma_{-}(\Sigma) & \sim \lambda^{\msharp_{-}}
\end{aligned}
\end{equation}
in $\Ring$, for unique natural numbers $\msharp_{+}$ and $\msharp_{-}$ depending
on $\Sigma$. We define
\[
         \msharp(\Sigma) = \frac{1}{2}(\msharp_{+} + \msharp_{-}) .
\]
From the relation \eqref{eq:pq}, it follows that $\msharp(\Sigma_{1})
= \msharp(\Sigma)+1$. Therefore, if we define
\[
      \ssharp = 2( g(\Sigma) - \msharp(\Sigma)),
\]
then $\ssharp$ is independent of $\Sigma$ and is an invariant of $K$
alone. This is our definition of the instanton variant
$\ssharp(K)$. The factor of $2$ is included only to match the
definition of $s(K)$ from \cite{Rasmussen}.
From the definition, it is clear that
\[
         \ssharp(K) \le 2 g_{*}(K)
\]
for any knot $K$, just as with Rasmussen's invariant.

If we write $\hat\Gamma$ for the local system $\Gamma\otimes_{\Sring}
\Ring$, then we can phrase the whole definition in terms of the
$\Ring$-module $\Isharp(K;\hat\Gamma)$. Let us abbreviate our notation
and write
\[
         I(K) = \Isharp(K;\hat\Gamma)
\]
and
\[
             I'(K) = I(K) / \mathrm{torsion}.
\]
Then the integers $\msharp_{+}$ and $\msharp_{-}$ corresponding to a
cobordism $\Sigma$ from $U$ to $K$ completely describe the map
$I'(U)\to I'(K)$ arising from $\Sigma$, just as in the Khovanov
case.

\subsection{Mirror images}

Our invariant changes sign when we replace $K$ by its mirror:

\begin{lemma}
    If $K^{\dag}$ is the mirror image of $K$, then $\ssharp(K^{\dag})
    = -\ssharp(K)$.
\end{lemma}

\begin{proof}
    The complex $C$ that computes $I(K^{\dag})$ can be identified with
    $\Hom(C,\Ring)$, so that $I(K^{\dag})$ is related to $I(K)$ as
    cohomology is to homology. The universal coefficient theorem
    allows us to identify
\[ 
        I'(K^{\dag}) = \Hom(I'(K) , \Ring).
\]
   Under this identification, the canonical mod-4 grading changes sign.
   Furthermore, if $\Sigma$ is a cobordism from $U$ to $K$, then we
   can regard it (with the same orientation) also as a cobordism $\Sigma^{\dag}$ from
   $K^{\dag}$ to $U$ and the resulting map $\psi^{\sharp}{\Sigma^{\dag}} :
   I'(K^{\dag}) \to I'(U)$ is the
   dual of the map $\psi^{\sharp}(\Sigma): I'(U) \to I'(K)$. So if $\zp$ and
   $\zm$ are standard generators for $I'(K^{\dag})$ in mod-4 gradings
   $1$ and $-1$, we have
   \[
   \begin{aligned}
       \psi^{\sharp}(\Sigma^{\dag}) : \zp &\mapsto \sigma_{-}(\Sigma) \up\\
       \psi^{\sharp}(\Sigma^{\dag}) : \zm &\mapsto \sigma_{+}(\Sigma) \um
   \end{aligned}
    \]
   if the genus is even or
   \[
   \begin{aligned}
       \psi^{\sharp}(\Sigma^{\dag}) : \zp &\mapsto \sigma_{-}(\Sigma) \um\\
       \psi^{\sharp}(\Sigma^{\dag}) : \zm &\mapsto \sigma_{+}(\Sigma) \up
   \end{aligned}
    \]
    if the genus is odd. Here $\sigma_{\pm}(\Sigma)$ are the elements
    of $\Sring$ associated earlier to the cobordism $\Sigma$ from $U$
    to $K$. 

   Now take any oriented cobordism $T$ from $U$ to $K^{\dag}$. For
   convenience arrange that $\Sigma$ and $T$ both have even genus. Consider the
   composite cobordism from $U$ to $U$. It is given by
   \[
   \begin{aligned}
       \up &\mapsto \sigma_{-}(\Sigma) \sigma_{+}(T) \up\\
       \um  &\mapsto \sigma_{+}(\Sigma) \sigma_{-}(T) \um.
   \end{aligned}
   \]
   For the composite cobordism, we therefore have
   \[
             \msharp(\Sigma \circ T) = \msharp(\Sigma) + \msharp(T),
   \]
   so
    \[
    \begin{aligned}
        0 &= \frac{1}{2} \ssharp(U) \\
          &= g(\Sigma\circ T) - \msharp(\Sigma \circ T) \\
          &= g(\Sigma) - \msharp(\Sigma) + g(T) - \msharp(T) \\
          &= \frac{1}{2} \ssharp(K^{\dag}) - \frac{1}{2} \ssharp(K).
    \end{aligned}
     \]
\end{proof}

\section{Equality of the two invariants}

To show that $\ssharp(K)=s(K)$ for all knots $K$, it will suffice to
prove the inequality
\begin{equation}\label{eq:the-inequality}
         \ssharp(K) \ge s(K)
\end{equation}
because both invariants change sign when $K$ is replaced by its mirror
image. 

\subsection{The cube and its filtration}

Let $D$ be a planar diagram for an oriented link $K$, and let the
resolutions $K_{v}$ and the Khovanov cube $\bC(D)$ be defined as in
section~\ref{subsec:Khovanov}. According to \cite{KM-unknot}, there is
a spectral sequence whose $E_{1}$ term is isomorphic to the Khovanov
complex corresponding to this diagram and which abuts to the instanton
homology $\Isharp$. The spectral sequence was defined in \cite{KM-unknot} using
$\Z$ coefficients, and we need to adapt it now to the case of the
local coefficient system $\Gamma$ or $\hat\Gamma$.

We continue to write $I(K)$ as an abbreviation for
$\Isharp(K;\hat\Gamma)$ and we write $\CI(K)$ for the complex that
computes it. The complex depends on choices of metrics and
perturbations. We form a cube $\bC^{\sharp}(D)$ from the instanton
chains complexes of all the resolutions:
\[
       \bC^{\sharp}(D) =  \bigoplus_{v} \CI(K_{v}).
\]
Although we previously used $\Z$ coefficients, we can carry over
verbatim from
\cite{KM-unknot} the definition of a differential $d^{\sharp}$ on
$\bC^{\sharp}(D)$ with components
\[
           d^{\sharp}_{vu} : \CI(K_{v}) \to \CI(K_{u})
\]
for all vertices $v\ge u$ on the cube. The homology of the cube
$(\bC^{\sharp}(D), d^{\sharp})$ is isomorphic to the instanton homology
$I(K)$: the proof is easily adapted from \cite{KM-unknot}.

As in \cite{KM-unknot}, the differential $d^{\sharp}$ respects the
decreasing filtration defined by the grading $h$ on cube, so we may
consider the associated spectral sequence.  We have the following:

\begin{proposition}
    The  page $(E_{1}, d_{1})$ of the spectral sequence associated to
    the filtered complex $(\bC^{\sharp}(D) , d^{\sharp})$ is
    isomorphic to the Khovanov cube $(\bC(D), d)$ in the version
    defined in section~\ref{subsec:Khovanov}.
\end{proposition}

\begin{proof}
    The definition of the filtered complex tells us that the $E_{1}$
    page is the sum of the instanton homologies of the unlinks
    $K_{v}$:
\[
               E_{1} = \bigoplus_{v\in\{0,1\}^{N}} I(K_{v}).
\]
    If $K_{v}$ is an unlink of $p(v)$ components, then the excision
    argument of \cite{KM-unknot} provides a natural isomorphism
    \[
           \gamma:  I(K_{v})   \to I(U)^{\otimes p(v)}
     \]
     where $U$ is the unknot. To compute $I(U)$, we return to the
     definitions: the critical points of the unperturbed Chern-Simons
     functional comprise an $S^{2}$, so $I(U)$ is a free
     $\Ring$-module of rank $2$, with generators differing by $2$ in
     the $\Z/4$ grading. Our conventions put these generators in
     degrees $1$ and $-1$. In order to nail down particular
     generators, we use the disk cobordisms $D_{+}$ from the empty
     link $U_{0}$ to the unknot $U=U_{1}$, and $D_{0}$ from $U_{1}$ to
     $U_{0}$. We define $\bv_{+} \in I(U_{1})$ to be the image of the
     canonical generator $1 \in \Ring \cong I(U_{0})$ under the map
     $\psi^{\sharp}(D_{+})$. 
     We define
     $\bv_{-}\in I(U_{1})$ to be the unique element with
     $\psi^{\sharp}(D_{-})(\bv_{-})=1$.

     At this point we have an isomorphism of $\Ring$-modules,
     \[
              E_{1} = \bigoplus_{v\in\{0,1\}^{N}} V^{\otimes p(v)},
     \]
     and so $E_{1} = \bC(D)$. The differential $d_{1}$ on the $E_{1}$
     page arises from the maps
     \[
                    \psi^{\sharp}_{vu} : I(K_{v}) \to I(K_{u})
     \]
      obtained from the elementary cobordisms $K_{v} \to K_{u}$
      corresponding to the edges of the cube. As in \cite{KM-unknot},
      and adopting the notation from there,
      we  need only examine the pair-of-pants cobordisms 
      \[
      \begin{aligned}
          \copants:  U_{2} &\to U_{1} \\
          \pants: U_{1} &\to U_{2}
      \end{aligned}
       \]
      and show that the corresponding maps
     \[
      \begin{aligned}
         \psi^{\sharp}( \copants) :& V\otimes V \to V \\
        \psi^{\sharp}( \pants) : &V \to V\otimes V
      \end{aligned}
     \]
     are equal respectively to the maps $\mprod$ and $\Delta$ from the
     definition of $\kh(K)$ in section~\ref{subsec:Khovanov}.

     Because of the mod $4$ grading, we know that we can write
     \[
     \begin{aligned}
         \psi^{\sharp}(\pants)(\bv_{+}) &= a (\bv_{+}\otimes \bv_{-}) + b
         (\bv_{+} \otimes
         \bv_{-}) \\
         \psi^{\sharp}(\pants)(\bv_{-}) &= c (\bv_{+}\otimes \bv_{+}) + d
         (\bv_{-} \otimes
         \bv_{-}) \\
     \end{aligned}
     \]
     for some coefficients $a$, $b$, $c$, $d$ in $\Ring$. Simple
     topological arguments, composing with the cobordisms $D_{+}$ and
     $D_{-}$, show that $a=b=d=1$. From Poincar\'e duality, then see
     that
     \[
     \begin{aligned}
         \psi^{\sharp}(\copants)(\bv_{+}\otimes \bv_{+}) &= \bv_{+} \\
         \psi^{\sharp}(\copants)(\bv_{+}\otimes \bv_{-}) &= \bv_{-} \\
         \psi^{\sharp}(\copants)(\bv_{-}\otimes \bv_{+}) &= \bv_{-} \\
         \psi^{\sharp}(\copants)(\bv_{-}\otimes \bv_{-}) &= c \bv_{+}.
     \end{aligned}     \]
      To compute the unknown element $c\in\Ring$ we compose $\pants$
      and $\copants$ to form a genus-1 corbodism from $U_{1}$ to
      $U_{1}$. In the basis $\bv_{+}$, $\bv_{-}$, the composite
      cobordism gives rise to the map
      \begin{equation}\label{eq:phi}
         \phi =  \begin{pmatrix}
                   0 & 2 c \\ 2 & 0
                 \end{pmatrix}.
      \end{equation}
   Comparing this with \eqref{eq:p-and-q}, we see that $q(u) = 2$ and
   $p(u)=2c$. We also have the relation \eqref{eq:pq}: in terms of the
   variable $\lambda$, this says $p(u) q(u) = 4\lambda^{2}$. We deduce
   that $c=\lambda^{2}$. Thus $\psi^{\sharp}(\copants)$ and $\psi^{\sharp}(\pants)$ are
   precisely the maps $\mprod$ and $\Delta$ from the definition of
   $\kh(K)$. This completes the proof.
\end{proof}

There is a further simplification we can make,
following~\cite[section~2]{KM-q-grading}. 
The critical point set
for the unperturbed Chern-Simons functional for the each unlink
$K_{v}$ is a product spheres, $(S^{2})^{p(v)}$. We can therefore
choose  a perturbation so that the complex $\CI(K_{v})$ has generators
whose indices are all equal mod $2$. The differential on $\CI(K_{v})$
is the zero, and the page $E_{1}$ is canonically isomorphic to
$\bC^{\sharp}(D)$ (i.e. to the $E_{0}$ page). 
At this point we can reinterpret the above
proposition in slightly different language:

\begin{proposition}\label{prop:cube-leading}
    With perturbations chosen as above, there is a canonical isomorphism of cubes
    $\bC^{\sharp}(D) \to \bC(D)$. Under this isomorphism, the
    differential $d^{\sharp}$ can be written as $d + x$, where $d$ is
    the Khovanov differential on $\bC(D)$ (which increases the $h$
    grading by $1$) and $x$ is a sum of terms all of which increase
    the $h$ grading by more than $1$. \qed
\end{proposition}

\begin{remark}
    The term $x$ has odd degree with respect to
    the mod $2$ grading defined by $h$.
\end{remark}

\subsection{Comparing maps from cobordisms}

Let $\Sigma$ be a cobordism from the unknot $U_{1}$ to a knot $K$. 
This cobordism provides maps
\[
\begin{aligned}
    \psi(\Sigma) : V &\to \kh(K) \\
    \psi^{\sharp}(\Sigma) : V &\to I(K).
\end{aligned}
\]
Given a diagram $D$ for $K$, we can represent both $\kh(K)$ and $I(K)$
as the homologies of a cube $\bC(D)= \bC^{\sharp}(D)$, with two different
differentials, as above.
At the chain level, the maps $\psi(\Sigma)$ and
$\psi^{\sharp}(\Sigma)$ can represented by maps to the cube:
\[
\begin{aligned}
    \Psi(\Sigma) : V &\to \bC(D) \\
    \Psi^{\sharp}(\Sigma) : V &\to \bC(D) .
\end{aligned}
\]
Since Khovanov homology is a graded theory with respect to $h$,
the map $\Psi(\Sigma)$ respects the $h$-grading, so the image of
$\Psi(\Sigma)$ lies in grading $h=0$. This does not hold of
$\Psi^{\sharp}(\Sigma)$. We shall show:

\begin{proposition}\label{prop:Psi-leading}
   For a suitable choice of $\Sigma$,
    the maps $\Psi(\Sigma)$ and $\Psi^{\sharp}(\Sigma)$ can be chosen
    at the chain level so that
     \[
            \Psi^{\sharp}(\Sigma) = a \Psi(\Sigma) + \bX,
     \]
      where $a\in\Ring$ is a unit and $\bX$ strictly increases the $h$-grading. 
\end{proposition}

\begin{remark}
    The authors presume that $a$ can be shown to be $\pm 1$, but the
    argument presented below will leave that open.
\end{remark}

\begin{proof}[Proof of the Proposition]
The construction of a chain map $\Psi^{\sharp}(\Sigma)$ at the level
of the cubes $\bC(D)$  is presented in \cite{KM-q-grading}. The
construction depends on decomposing $\Sigma$ as a composite of
elementary cobordisms, each of which corresponds either to a
Reidemeister move or to the addition of handle of index $0$, $1$ or
$2$. The map $\Psi(\Sigma)$ on the Khovanov complex is constructed in
a similar manner. So it is sufficient to consider such elementary
cobordisms. Thus we consider two knots or links $K_{1}$ and $K_{0}$,
with diagrams $D_{1}$ and $D_{0}$, differing either by a Reidemeister
move or by a handle-addition, and a standard oriented cobordism $\Sigma$
between them. (The cobordism is topologically a cylinder in the case
of a Reidemeister move.) There is a chain map
\[
        \Psi^{\sharp}(\Sigma)  : \bC(D_{1}) \to \bC(D_{0})
\]
representing the map $\psi^{\sharp}(\Sigma)$.

  In \cite{KM-q-grading} it is shown how to define
$\Psi^{\sharp}(\Sigma)$ so that it 
respects the decreasing filtration defined by
$h$. (See Proposition~1.5 of \cite{KM-q-grading}. The term $S\cdot S$
is absent in our case because our cobordism is orientable.) So we can
write
\[
        \Psi^{\sharp}(\Sigma) = \Psi_{0} + \bX
\]
where $\Psi_{0}$ respects the $h$-grading, and $\bX$ strictly
increases it. From Proposition~\ref{prop:cube-leading} we see that
  $\Psi_{0}$ is a chain map for the Khovanov differential $d$. 
If we check that $\Psi_{0}$ induces the map
  $a \psi(\Sigma)$ on Khovanov homology, for some unit $a$, then we will be done.
    \begin{figure}
        \centering
        \includegraphics[height=1.0 in]{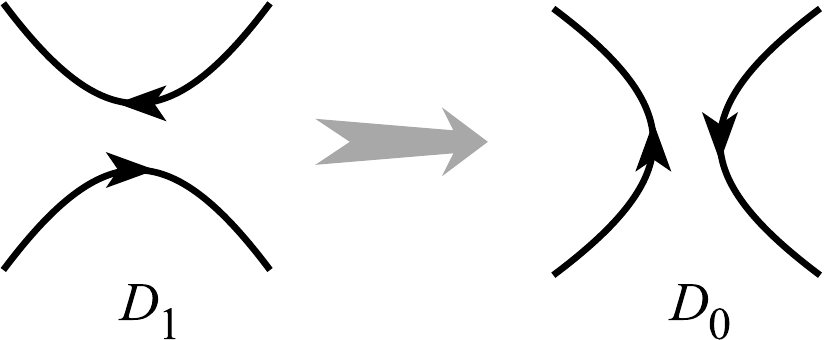}
        \caption{Changing an oriented link diagram by the addition of an oriented 1-handle.}
        \label{fig:1-handle}
    \end{figure}

Since this verification seems a little awkward for the Reidemeister
III move in particular, we simplify the work needed here
e by exploiting the fact that
we are free to choose $\Sigma$ as we wish. We start with a
crossing-less diagram for the unknot $U_{1}$ and then perform
$1$-handle additions to this diagram to obtain an unlink with many
components. Next we perform Reidemeister I moves to components of this
unlink, at most one per component, to introduce some crossings. We
then perform more $1$-handle additions on the diagram. In this way we
can arrive at a cobordism from $U_{1}$ to any knot $K$, using only
$1$-handles together with Reidemeister I moves applied to circles
without crossings in the diagram.

In fact, because we need to compare the leading term $\Psi_{0}$ with
the map $\psi(\Sigma)$ defined in \cite{Rasmussen}, we should treat
only the \emph{negative} Reidemeister I move directly: the positive
version, Reidemeister I$_+$, can be obtained as the composite of a
Reidemeister II and the inverse of a Reidemeister I$_-$.

With this simplification, we see that we need only check two cases:
\begin{enumerate}
\item \label{item:RI} introducing a single negative crossing in a circle without crossings, by
    a Reidemeister I$_-$ move;
\item  \label{item:RII} introducing two crossings in a circle without crossings, by
    a Reidemeister II move;
\item  \label{item:1-handle} altering the diagram by the addition of an oriented $1$-handle,
    as shown in Figure~\ref{fig:1-handle}. 
\end{enumerate}

Case \ref{item:1-handle} is the most straightforward of these
three. If $D_{1}$ and $D_{0}$ are diagrams as shown in the figure,
then we can draw a third diagram $D_{2}$ representing a link with an
extra crossing, so that $D_{1}$ and $D_{0}$ are the two diagrams
obtained by smoothing the one extra crossing in the two standard
ways. The cube $\bC(D_{2})$ can then be regarded as $\bC(D_{1}) \oplus
\bC(D_{0})$, and the map $\Psi^{\sharp}(\Sigma): \bC(D_{1}) \to
\bC(D_{0})$ corresponding to the addition of a 1-handle is equal to
the corresponding component of the differential on $\bC(D_{2})$. (See
\cite[section~10]{KM-q-grading}.) The
same is true for the corresponding map in the Khovanov theory, as
described in \cite{Rasmussen}; so the
fact that the leading term of $\Psi^{\sharp}(\Sigma)$ is the Khovanov
map is just an instance of Proposition~\ref{prop:cube-leading},
applied to the diagram $D_{2}$.

Cases \ref{item:RI} and \ref{item:RII} can be treated together, as follows.
Since the Reidemeister move only involves a single circle without
crossings, we can pass to the simple case in which $D_{1}$ itself is a
single circle. This is because both the instanton theory and the
Khovanov theory have tensor product rule for split diagrams, and it is
easy to check that the two product rules agree to leading order. In
the case of a single circle, we have then two chain maps
\[
\begin{aligned}
    \Psi_{0} : V &\to \bC(D_{0}, d) \\
    \Psi    : V &\to \bC(D_{0}, d) \\
\end{aligned}
\]
where the first map is the leading term of $\Psi^{\sharp}$ and the
second map is the one defined for the Khovanov theory in
\cite{Rasmussen}. The maps $\psi_{0}$ and $\psi$ which these give rise
to on homology are both isomorphisms, from one free module of rank $2$
to another. Since both preserve the same grading mod $4$, it follows
that $\psi_{0}$ and $\psi$ differ at most by an automorphism $\delta$
of $V$ of the form
\[
\begin{aligned}
    \vp &\mapsto a \vp \\
    \vm &\mapsto b \vm ,
\end{aligned}
\]
where $a$ and $b$ are units in $\Ring$. We must only check that
$a=b$. Because we have already dealt with $1$-handle additions, we can
see that $a=b$ as a consequence of the fact that $\delta$ must commute
with the map $\phi$ from \eqref{eq:phi}. This map corresponds a
cobordism from $U_{1}$ to itself genus $1$, and can be realized as the
addition of two $1$-handles. Both $\psi_{0}$ and $\psi$ intertwine the
map $\phi : V \to V$ by a map on the homology of $\bC(D_{2}, d)$
arising from a chain map
\[
         \Phi :\bC(D_{2}, d) \to \bC(D_{2}, d)
\]
given (in both cases) by applying $\phi$ or $\phi\otimes 1$ as
appropriate at each vertex of the cube. In the case of $\psi_{0}$,
this follows from the naturality of instanton homology, while in the
case of $\psi$ it can be verified directly.
\end{proof}

The inequality \eqref{eq:the-inequality} -- and hence
the equality \eqref{eq:1} -- can be deduced from
Proposition~\ref{prop:Psi-leading} as follows. From the definitions of
and $s$ and $\ssharp$, it is apparent that it will suffice to show
\[
\begin{aligned}
    m^{\sharp}_{+} &\le m_{+} \\
    m^{\sharp}_{-} &\le m_{-},
\end{aligned}
\]
where $m^{\sharp}_{+}$ etc.~are the numbers associated to the
cobordism $\Sigma$. Consider the first of these two
inequalities. (The second is not essentially different.) Although the
definition of $m^{\sharp}_{+}$ was in terms of $I'(K) =
I(K)/\mathrm{torsion}$, we can use the fact that all the torsion is
annihilated by some power of $\lambda$ to rephrase the definition:
$m^{\sharp}_{+}$ is the largest integer $m$ such that, for all
sufficiently large $k$, the element
\[
              a^{\sharp}_{k}  := \psi^{\sharp}(\Sigma)(\lambda^{k} \bv_{+})
\]
is divisible by $\lambda^{k+m}$. This means that the
chain representative
\[
        \ba^{\sharp}_{k} = \Psi^{\sharp}(\Sigma)(\lambda^{k} \bv_{+})
\]
is $d^{\sharp}$-homologous to an element of $\bC(D)$ that is divisible
by $\lambda^{k+m}$:
\begin{equation}\label{eq:ak}
      \ba^{\sharp}_{k} = \lambda^{k+m} \bb^{\sharp} +
      d^{\sharp} \bc^{\sharp}
\end{equation}
for $m=m^{\sharp}_{+}$.

Let us write $\cF^{i}$ for the subcomplex of $\bC(D)$ generated by
the elements in $h$-grading $\ge i$. In \eqref{eq:ak}, the element
$\ba^{\sharp}_{k}$ on the left belongs to $\cF^{0}$ and its leading
term in degree $h=0$ is equal to 
\[
   \ba_{k} = \Psi(\Sigma)(\lambda^{k}\bv_{+})
\]
by Proposition~\ref{prop:Psi-leading}. A priori, we do not know which
level of the filtration the elements $\bb^{\sharp}$ and $\bc^{\sharp}$
belong to. But we have the following lemma.

\begin{lemma}
    After perhaps increasing $k$, we can arrange
that
\begin{equation}\label{eq:ak-repeat}
 \ba^{\sharp}_{k} = \lambda^{k+m} \bb^{\sharp} +
      d^{\sharp} \bc^{\sharp} 
\end{equation}
where
$\bb^{\sharp}$ lies in $\cF^{0}$ and $\bc^{\sharp}$ lies in
$\cF^{-1}$. 
\end{lemma}

\begin{proof}
    We deal with $\bb^{\sharp}$ first. Suppose that $\bb^{\sharp}$
    does not lie in $\cF^{0}$ but lies in $\cF^{-n}$ for some $n>0$,
    with a non-zero leading term $b_{-n}$:
    \[
              \bb^{\sharp} = b_{-n} + z.
    \]
    Writing $d^{\sharp} = d + x$ as before, the degree-1 term in the
    equality $d^{\sharp} \bb^{\sharp} = 0$ tells us that $d
    b_{-n}=0$. Since the torsion-free part of $H(\bC(D), d)$ is
    supported in $h$-grading $0$, the cycle $\lambda^{l}b_{-n}$ must
    be $d$-exact once $l$ is large enough:
     \[
         \lambda^{l} b_{-n} = d\beta
     \]
     for some $\beta$ in $h$-grading $-n-1$.  If we define
      \[
            \tilde\bb^{\sharp} = \lambda^{l} \bb^{\sharp} - d^{\sharp}\beta,
      \]
      then $\tilde\bb^{\sharp}$ belongs to $\cF^{-n+1}$ and is
      $d^{\sharp}$-homologous to $\ba^{\sharp}_{k+l}$. We replace $k$
      by $k+l$, and replace $\bb^{\sharp}$ with $\tilde\bb^{\sharp}$.
      The effect is to decrease $n$ by at least $1$. In this way
      we eventually arrive at a situation with $\bb^{\sharp} \in
      \cF^{0}$.

       Once we have $\bb^{\sharp}$ in $\cF^{0}$, a similar argument
       applies to $\bc^{\sharp}$. Suppose that \eqref{eq:ak-repeat}
       holds, with $\bb^{\sharp}$ in $\cF^{0}$ and $\bc^{\sharp}$ 
       in $\cF^{-n}$ with a non-zero leading term $c_{-n}$ in
       $h$-grading $-n$, for some $n>1$. In \eqref{eq:ak-repeat}, the
       only term in degree $-n+1$ is $d c_{-n}$, so $d c_{-n}$ is
       zero. Using again the fact that the $d$-homology in negative
       $h$-gradings is torsion, we conclude that $\lambda^{l} c_{-n}$
       is $d$-exact, for some $l$:
         \[
                            \lambda^{l} c_{-n} = d \gamma.
          \]
       If we set \[ \tilde\bc^{\sharp} = \lambda^{l} \bc^{\sharp} -
       d^{\sharp} \gamma, \]
       then we have $d^{\sharp} \tilde \bc^{\sharp} = \lambda^{l}
       d^{\sharp}\bc^{\sharp}$, so
        \[
             \ba_{k+l}^{\sharp} =  \lambda^{l+k+m} \bb^{\sharp} +
      d^{\sharp} \tilde \bc^{\sharp} .
        \]
        On the other hand, $\tilde\bc^{\sharp}$ lies in $\cF^{-n+1}$,
        so we are better off by at least $1$. Eventually, we arrive in $\cF^{-1}$.        
\end{proof}

Given $\ba^{\sharp}_{k}$, $\bb^{\sharp}$ and $\bc^{\sharp}$ as in the lemma, let us write
$\ba_{k}$, $\bb$ and $\bc$ for the components of these in $h$-gradings
$0$, $0$ and $-1$ respectively. We have $\ba_{k} =
\Psi(\Sigma)(\lambda^{k}\bv_{+})$ as just noted,
and the leading term of the identity in the lemma is
\[
          \ba_{k} = \lambda^{k+m} \bb + d \bc.
\]
So in the homology group $\kh(K) = H(\bC(D), d)$, the element 
$\psi(\Sigma)(\lambda^{k}\bv_{+})$ is represented at the chain level
by a chain that is divisible by $\lambda^{k+m}$, where $m =
m^{\sharp}_{+}$. The largest possible power of $\lambda$ by which a
chain-representative of this class can be divisible is $m_{+}$,
by definition, so $m_{+} \ge m^{\sharp}_{+}$. This completes the proof
that $\ssharp(K) \ge s(K)$.

\section{Negative-definite manifolds and connected sums}

We give the proof of Corollary~\ref{cor:neg-def}.

\subsection{Three lemmas}

 As mentioned earlier and proved in
\cite{KM-unknot} and \cite{KM-singular}, the invariant $\Isharp(K)$ is
functorial not just for cobordisms in $[0,1]\times S^{3}$, but more
generally for oriented cobordisms of pairs $(W,\Sigma)$, where $W$ is
a 4-dimensional cobordism from $S^{3}$ to $S^{3}$ and $\Sigma$ is an
embedded 2-dimensional cobordism. We can also extend the definition to
immersed cobordisms $\Sigma \looparrowright W$ with normal crossings,
just as we did before, by blowing up and taking proper transforms.

Let $K$ be a knot in $S^{3}$, and let $(W,\Sigma)$ be such an oriented
cobordism of pairs (with $\Sigma$ embedded, not immersed), 
from $(S^{3}, U)$ to $(S^{3},\Sigma)$. Using the
induced map $\psi(W, \Sigma)$ to $\Isharp(K;\Gamma)$, we define
elements
\[
        \sigma_{+}(W,\Sigma) , \sigma_{-}(W,\Sigma) \in \Sring,
\]
well-defined up to multiplication by units, just as we did at
\eqref{eq:sigma-pm}. Passing to the completion $\Ring$, we obtain
non-negative numbers $\msharp_{-}(W,\Sigma)$ and
$\msharp_{+}(W,\Sigma)$. We take these to be $+\infty$ is
$\sigma_{\pm}(W,\Sigma)$ is zero. We then define
\[
      \msharp(W,\Sigma) = \frac{1}{2}(\msharp_{+}(W,\Sigma) +
      \msharp_{-}(W,\Sigma))
\]
and
\[
          \ssharp(W,\Sigma) = 2(g(\Sigma) - \msharp(W,\Sigma)),
\]
as we did in the case that $W$ was a cylinder. The proofs from the
previous case carry over, to establish the following to lemmas.

\begin{lemma}
    If $\Sigma$ and $\Sigma'$ are two embedded cobordisms in the same
    $W$ and if $\Sigma$ is homotopic to $\Sigma'$ relative to their
    common boundary, then $\ssharp(W,\Sigma) = \ssharp(W,\Sigma')$.
     \qed
\end{lemma}

\begin{lemma}
    If $\Sigma_{1}$ is obtained from $\Sigma$ by forming the connected
    sum with an embedded torus inside a standard $4$-ball in $W$, then
    $\ssharp(W,\Sigma) = \ssharp(W,\Sigma_{1})$. \qed
\end{lemma}

The third lemma that we shall need is a connected sum theorem for $W$:

\begin{lemma}
    Suppose that $W'$ is obtained from $W$ as a connected sum $W'=W\#
    N$, with the sum being made in a ball in $W$ disjoint from
    $\Sigma$. Suppose that $N$ has $b_{1}(N)=0$ and
    $b_{2}^{+}(N)=0$. Then $\ssharp(W',\Sigma)  = \ssharp(W,\Sigma)$.
\end{lemma}

\begin{proof}
    The essential parts of this connect-sum theorem are proved in
    \cite{Donaldson-orientations}. We are not assuming that $N$ is
    simply connected, so the instanton moduli space $M_{0}(N)$ for
    instanton number $0$ (i.e. the space of flat $\SU(2)$ connections)
    is potentially non-trivial. After choosing suitable holonomy
    perturbations on $N$ as in \cite{Donaldson-orientations}, we may
    arrange that the perturbed moduli space consists of the following:
    \begin{enumerate}
    \item one reducible flat connection for each homomorphism
        $\pi_{1}(N) \to \{\pm 1 \} \subset \SU(2)$;
    \item one $S^{1}$-reducible connection for each of the finitely many $S^{1}$-reducible
        flat connections on $N$.
    \end{enumerate}
    Furthermore, for the perturbed-flat connections $A$ of the second
    kind, we can be assured that the kernel of the linearized
    perturbed equations with gauge-fixing is zero.

    We now follow the standard argument. We consider a $1$-parameter
    family of metrics on $W'$ in which the neck of the connected sum
    is stretched. The usual dimension-counting shows that the only
    contributions come from flat (or perturbed-flat) connections on
    $N$. Following the analysis of the orientations from
    \cite{Donaldson-orientations}, we see that
      \[
               \psi(W',\Sigma) = (a + 2b) \psi(W,\Sigma)
       \]
      where $a$ is the number of points in $M_{0}(N)$ of the first
      kind and $b$ is the number of points of the second kind. The
      quantity $a+2b$ is the order of $H_{1}(N;\Z)$. Since in any
      event this number is non-zero, the result follows.
\end{proof}

\subsection{Conclusion of the proof}

As a corollary of these three lemmas, we deduce:

\begin{corollary}
    If $W$ has $b_{1}(W)=0$ and $b_{2}^{+}(W)=0$, and if the inclusion
    $\Sigma\hookrightarrow W$ is homotopic to  map whose image is
    contained in the union of $\partial W= S^{3}\cup S^{3}$ together
    with an arc joining them, then $\ssharp(W,\Sigma)$ coincides with
    the standard invariant $\ssharp(K)$.
\end{corollary}

\begin{proof}
    After adding additional handles to $\Sigma$ (and applying the
    second of the three lemmas above), we may assume that $\Sigma$ is
    homotopic to an embedded surface $\Sigma'$ entirely contained in a
    neighborhood of $S^{3} \cup S^{3}$ and the arc joining them. We
    may replace $\Sigma$ with $\Sigma'$, because of the first
    lemma. At this point, we can regard $W$ as a connected sum of a
    standard cylinder $[0,1]\times S^{3}$ with a negative definite
    manifold $N$. Our surface is disjoint from $N$, so the third lemma
    applies. That is, $\ssharp(W,\Sigma)$ is equal to
    $\ssharp([0,1]\times S^{3}, \Sigma)$. The latter is, by
    definition, the invariant $\ssharp(K)$.
\end{proof}

\begin{corollary}
    Let $X$ be an oriented $4$-manifold with boundary $S^{3}$ that has $b_{1}(X)=b_{2}^{+}(X)=0$, and let
    $\Sigma$ be a properly embedded, connected oriented surface with
    boundary a knot $K$ in $S^{3}$. Suppose that the inclusion of
    $\Sigma$ in $X$ is homotopic to the boundary $S^{3}$ relative to
    $K$. Then $2 g(\Sigma) \ge \ssharp(K)$.
\end{corollary}

\begin{proof}
    We may regard $(X,\Sigma)$ as providing a cobordism $(W,\Sigma')$
    by removing a standard ball-pair. From the definition of
    $\ssharp(W,\Sigma')$, we have $\ssharp(W,\Sigma') \le 2
    g(\Sigma)$. But $\ssharp(W,\Sigma') = \ssharp(K)$ by the previous corollary.
\end{proof}

Corollary~\ref{cor:neg-def} in the introduction is simply a
restatement of the last corollary above, in light of the equality $\ssharp=s$.

\bibliographystyle{abbrv}
\bibliography{s-invariant}

\end{document}


\pagestyle{empty}
\parskip=8pt

\maketitle

The authors' paper \cite{KM:s-invariant} constructs an integer-valued
invariant, $s^{\sharp}(K)$, for knots $K$ in $S^{3}$, and establishes
the following properties:
\begin{enumerate}
\item $s^{\sharp}(-K) = s^{\sharp}(K)$ if $-K$ is the mirror image of $K$;
 \item $|s^{\sharp}(K)/2|$ is a lower bound for the slice genus of $K$.
 \end{enumerate}
 
\noindent
 As a generalization of (2), the following is also
established in \cite{KM:s-invariant}.

\begin{enumerate}
    \setcounter{enumi}{2}
\item Suppose $X$ is a smooth, compact, oriented 4-manifold with boundary
    $S^{3}$, having negative-definite intersection form and $b_{1}=0$.
    Suppose $\Sigma\subset X$ is a properly embedded oriented surface with boundary the
    knot $K \subset S^{3}$, and suppose that the inclusion of $\Sigma$
    in $X$ is homotopic relative to $K$ to a map to $S^{3}$. Then
    $$g(\Sigma) \ge |s^{\sharp}(K)/2|.$$
    In particular, this inequality holds for a surface $\Sigma$
    embedded in a homotopy 4-ball.
 \end{enumerate}

 These results stand. But
 the paper also claims, incorrectly, that $s^{\sharp}(K)$ is equal to Rasmussen's
 invariant $s(K)$, defined in \cite{Rasmussen}. The proof 
 of Lemma~4.4 in \cite{KM:s-invariant}, on which this claim rests, contains an error.
 Consequently, it is an open question whether Rasmussen's invariant $s$
 shares property (3). There remains the possibility (explored in
 \cite{FGMW}) that one might detect an exotic $4$-ball $X$ by the
 existence of an embedded surface $\Sigma\subset X$ for which the
 inequality
 \[
       g(\Sigma) \ge |s(K)/2|
 \]    
does not hold.

The authors are grateful to Sherry Gong for pointing out that
$s^{\sharp}$ and $s$ must differ even for the trefoil. It is also
worth noting here that (unlike $s$) the invariant $s^{\sharp}$ is not
additive for connected sums of knots. We hope to return to these
observations in a future paper.

\bibliographystyle{lms-plain}
\bibliography{s-invariant-corrigendum}